\documentclass[11pt,draft]{amsart}
\usepackage{amssymb, amstext, amscd, amsmath}
\usepackage[all]{xy}
\usepackage{centernot}
\usepackage{mathtools}
\makeatletter
\def\@cite#1#2{{\m@th\upshape\bfseries%
[{#1\if@tempswa{\m@th\upshape\mdseries, #2}\fi}]}}
\makeatother
\theoremstyle{plain}
\newtheorem{theorem}{Theorem}[section]

\newtheorem{proposition}[theorem]{Proposition}
\newtheorem{lemma}[theorem]{Lemma}

\theoremstyle{definition}
\newtheorem{definition}[theorem]{Definition}

\newtheorem{remark}[theorem]{Remark}
\theoremstyle{remark}
\newcommand{\bbC}{{\mathbb{C}}}

\newcommand{\bbI}{{\mathbb{I}}}

\newcommand{\bbN}{{\mathbb{N}}}

\newcommand{\A}{{\mathcal{A}}}
\newcommand{\B}{{\mathcal{B}}}
\newcommand{\C}{{\mathcal{C}}}

\newcommand{\G}{{\mathcal{G}}}
\renewcommand{\H}{{\mathcal{H}}}

\newcommand{\J}{{\mathcal{J}}}
\newcommand{\K}{{\mathcal{K}}}

\renewcommand{\O}{{\mathcal{O}}}

\newcommand{\T}{{\mathcal{T}}}

\newcommand{\X}{{\mathcal{X}}}

\renewcommand{\phi}{\varphi}
\newcommand{\upchi}{{\raise.35ex\hbox{\ensuremath{\chi}}}}

\newcommand{\Aut}{\operatorname{Aut}}

\newcommand{\id}{{\operatorname{id}}}

\newcommand{\ca}{\mathrm{C}^*}
\newcommand{\cenv}{\mathrm{C}^*_{\text{env}}}
\newcommand{\cmax}{\mathrm{C}^*_{\text{max}}}

\newcommand{\sot}{\textsc{sot-}}

\newcommand{\sca}[1]{\left\langle#1\right\rangle}
\newcommand\cpr{\rtimes_{\alpha}^{r}\, {\mathcal{G}}}
\newcommand\cpf{\rtimes_{\alpha}\, {\mathcal{G}}}

\newcommand\cpd{ \hat{\rtimes}_{\alpha}\, {\mathcal{G}}}

\begin{document}

\title[Discrete Crossed Products]{$\ca$-envelopes and the Hao-Ng Isomorphism for discrete groups}

\author[E.G. Katsoulis]{Elias~G.~Katsoulis}
\address {Department of Mathematics
\\East Carolina University\\ Greenville, NC 27858\\USA}
\email{katsoulise@ecu.edu}


\maketitle

%%%%%%%%%%%%%%%%
\begin{abstract}
We establish the isomorphism
\[
\O_X \cpr \simeq \O_{X\cpr} 
\]
for any non-degenerate $\ca$-correspondence $(X, \C)$ and any discrete group $\G$ acting on $(X, \C)$. In the case where $(X, \C)$ is hyperrigid, a similar formula is also established for the full crossed product. These results are obtained by studying the $\ca$-envelopes and the hyperrigidity of a related class of non-selfadjoint crossed products that were introduced recently by the author and Chris Ramsey.
\end{abstract}

\section{introduction}
 Let $\G$ be a locally compact group acting on a $\ca$-correspondence $(X, \C)$. By the universality of the Cuntz-Pimsner $\ca$-algebra $\O_X$,  $\G$ also acts on $\O_X$. The Hao-Ng Theorem \cite[Theorem 2.10]{HN} asserts that
\[
\O_X \cpf \simeq \O_{X\cpf} 
\]
provided that $\G$ is an amenable locally compact group. This elegant result is having an increasing impact on current $\ca$-algebra research \cite{Ab, Deaconu, DKQ, Sch}. The \textit{Hao-Ng isomorphism problem} asks whether the above isomorphism remains valid, for either the full or the reduced crossed product, if one moves beyond the class of amenable groups. Given the important role that the crossed products of $\ca$-correspondences by group actions play in $\ca$-algebra theory, e.g. \cite{Es, Kasparov}, it is not surprising that the Hao-Ng isomorphism problem is currently under investigation by several $\ca$-algebraists \cite{BKQR, KQR, KQR2, Kim, Morgan}. We note here that even the 

The primary purpose of this paper is to provide a positive resolution of the Hao-Ng isomorphism problem for all discrete groups. Specifically, in Theorem~\ref{HaoNgreduced} we verify the Hao-Ng isomorphism for the reduced crossed product of any non-degenerate $\ca$-correspondence $(X, \C)$ by any discrete group $\G$. This result has been sought after by others as it was previously obtained only in special cases. (See in particular \cite[Theorem 5.5]{BKQR} where the authors deal with case of an \textit{exact} discrete group.) We also investigate the full crossed product and we obtain similar results there: Theorem~\ref{HaoNgfull} shows that the Hao-Ng isomorphism holds for the full crossed product of a \textit{hyperrigid} $\ca$-correspondence by any discrete group.  In particular this applies to the crossed product of a (row finite) graph correspondence by any discrete group. (See below for definitions.)

Our contribution to the Hao-Ng isomorphism problem stems from non--selfadjoint considerations. Indeed in our recent paper \cite{KR}, Chris Ramsey and the author developed a theory of crossed products that allows for a locally compact group to act on an arbitrary operator algebra, not just a $\ca$-algebra. In \cite[Section 7]{KR} we made a good case that the Hao-Ng isomorphism problem is intimately related to the positive resolution of the crossed product identities
\begin{equation} \label{cenvfull}
\cenv\big( \A \cpf \big) \simeq \cenv(\A) \cpf,
\end{equation}
and
\begin{equation} \label{cenvred}
\cenv\big( \A \cpr \big) \simeq \cenv(\A) \cpr,
\end{equation}
where $(\A , \G, \alpha)$ denotes a (not necessarily self adjoint) dynamical system. This is also the approach that we adopt here. As it turns out, the central result of this paper, Theorem~\ref{discrenv}, verifies the identity (\ref{cenvred}) in the case where $\A$ is any approximately unital operator algebra and $\G$ any discrete group. However the identity (\ref{cenvfull}) is not  being verified here. Instead we verify a similar identity for a related crossed product in the case where $\G$ is discrete and $\A$ is hyperrigid. This suffices to establish a ``Hao-Ng type isomorphism" for the full crossed product of a hyperrigid $\ca$-correspondence by a discrete group.

Hyperrigidity plays an important role in this paper. In the last section of the paper we strengthen the ties between hyperrigidity and the theory of crossed products by showing the permanence of hyperrigidity under the reduced crossed product (Theorem~\ref{hr}). We also indicate the permanence of hyperrigidity under the relative full crossed product associated with the $\ca$-envelope of an operator algebra; see Theorem~\ref{hrf}.  Both results increase the supply of hyperrigid operator algebras and therefore the applicability of the results of the earlier sections.

We hope that the reader will appreciate both our contribution to the Hao-Ng isomorphism problem \textit{and} the techniques involved in achieving it. Contrary to one might have expected from previous considerations, the Hao-Ng isomorphism problem, a seemingly selfadjoint problem, is very amenable to non-selfadjoint techniques. This places the present work right at the crossroads of the selfadjoint and non-selfadjoint operator algebra theory. It goes without saying that the paper should be of interest not only to the practitioners of the specific problem it addresses but also to any operator algebraist interested in the way the selfadjoint and the non-selfadjoint theories interact with each other.

%%%%%%%%%%%%%%%%%%%%%%%%%%%%%%%%%%%%
%%%%%%%%%%%%%%%%%%%%%%%%%%%%%%%%%%%%%%%%%%%%%%%%%

\section{Crossed products and their $\ca$-envelopes}

In this paper all operator algebras are assumed to be approximately unital, i.e., they have a contractive approximate unit. Nevertheless on occasion we will need to exploit the richer structure of unital operator algebras.  If $\A$ is an operator algebra without a unit, let $\A^1 \equiv \A + \bbC I$. If $\phi : \A \rightarrow \B$ is a completely isometric homomorphism between non-unital operator algebras, then Meyer \cite{Mey} shows that $\phi$ extends to a complete isometry $\phi^1: \A^1 \rightarrow \B^1$. This shows that the unitization of $\A$ is unique up to complete isometry.

Given an operator algebra $\A$, a $\ca$-cover $(\C, j)$ for $\A$ consists of a $\ca$-algebra $\C$ and a completely isometric multiplicative injection $ j : \, \A \rightarrow \C$ with $\C = \ca(j(\A))$. There are two distinguished $\ca$-covers associated with any (approximately unital) operator algebra $\A$.

The $\ca$-envelope $\cenv(\A) \equiv(\cenv(\A), j)$ of $\A$ is the universal $\ca$-cover of $\A$ with the following property: for any cover $(\C , i)$ of $\A$  there exists a $*$-epimorphism $\phi : \, \C \rightarrow \cenv(\A)$ so that $\phi(i(a)) = j(a)$, for all $a \in \A$. (See \cite[Proposition 4.3.5]{BlLM}.)

If $\A$ is an operator algebra then there exists a $\ca$-cover $\cmax(\A) \equiv ( \cmax(\A) , j)$ with the following universal property: if $\pi: \A \rightarrow \C$ is any completely contractive homomorphism into a $\ca$-algebra $\C$, then there exists a (necessarily unique) $*$-homomorphism $\phi: \cmax(\A) \rightarrow \C$ such that $\phi\circ j = \pi$. The cover $\cmax(\A)$ is called the maximal or universal $\ca$-algebra of $\A$. (See \cite[Proposition 2.4.2]{BlLM}.)

All groups appearing in this paper are thought to be discrete. A discrete dynamical system $(\A, \G, \alpha)$ consists of an approximately unital operator algebra $\A$ and a discrete group $\G$  acting on $\A$ by completely isometric automorphisms, i.e., there exists a group representation $\alpha: \G \rightarrow \Aut\A$. (Here $\Aut \A$  denotes the collection of all completely isometric automorphisms of $\A$.) Let $(\A, \G, \alpha)$ be a discrete dynamical system and let $(\C, j)$ be a $\ca$-cover of $\A$. Then $(\C, j)$ is said to be $\alpha$-admissible, if there exists a group representation $\dot{\alpha}: \G \rightarrow \Aut(\C)$ which extends the representation
  \begin{equation} \label{detailona}
  \G \ni s \mapsto j\circ \alpha_s \circ j^{-1} \in \Aut (j(\A)).
  \end{equation}
Since $\dot{\alpha}$ is uniquely determined by its action on $j(\A)$, both (\ref{detailona}) and its extension $\dot{\alpha}$ will be denoted by the symbol $\alpha$.

As we show in \cite[Lemma 3.3]{KR}, for any dynamical system $(\A, \G, \alpha)$, there are at least two $\alpha$-admissible $\ca$-covers, $(\cenv( \A), j)$ and $(\cmax (\A), j)$.

\begin{definition}[Relative Crossed Product] \label{relative}
Let $(\A, \G, \alpha)$ be a dynamical system and let $(\C, j)$ be an $\alpha$-admissible $\ca$-cover for $\A$. Then, $\A \rtimes_{\C, j, \alpha} \G$ and $\A \rtimes_{\C, j, \alpha}^r \G$ will denote the closed subalgebras of the crossed product $\ca$-algebras $\C \rtimes_{\alpha} \G$ and $\C \rtimes_{\alpha}^r \G$ respectively, which are generated by all finite sums of the form $\sum_{g \in \G} a_gu_g$, with $a_g \in \G$ and $u_g$ being the ``universal" unitaries implementing $\alpha_g$, $g \in \G$.
\end{definition}

As it turns out \cite[Theorem 3.12]{KR}, all reduced crossed products coincide in a canonical way. We therefore have a unique object for the reduced crossed product which we denote as $\A \cpr$. In the case of the full crossed product however the situation remains mysterious and we do not know yet if all such crossed products coincide. 

\begin{definition}[Full Crossed Product] \label{fulldefn}
If $(\A, \G, \alpha)$ is a dynamical system then
\[
\A \rtimes_{\alpha} \G \equiv  \A \rtimes_{\cmax(\A), \alpha} \G
\]
\end{definition}

In \cite{KR} we make a good case that the above definition is the ``right one" for the full crossed product. Indeed, in \cite[Proposition 3.7]{KR} we show that $\A\cpf$ is the universal algebra for all covariant representations of the system $(\A, \G, \alpha)$, i.e., the completely contractive representations of $\A \cpf$ coincide with the collection of all integrated forms of covariant representations for $(\A, \G, \alpha)$.

Nevertheless in this paper we will not be able to extract any useful information from $\A \cpf$  regarding the Hao-Ng isomorphism problem for the full crossed product. As it turns out the ``right" crossed product for us is $\A \rtimes_{\cenv(\A), \alpha} \G$. Note however that it is an open problem of \cite{KR} whether these two crossed products coincide.

In order to establish our Hao-Ng isomorphism for discrete groups, we will need to identify the $\ca$-envelopes of the various crossed products discussed above. We briefly review the results required from the pertinent theory.

 Let $\A$ be a unital operator algebra (actually operator space will suffice for this paragraph) and $\pi:  \A \rightarrow B(\H)$ be a unital completely contractive map. A \textit{dilation} $\rho:\A \rightarrow B(\K)$ for $\pi$ is another unital completely contractive map satisfying $P_{\H} \rho(.)\mid_{\H}= \pi$; such a dilation $\rho$ for $\pi$ is called trivial if $\rho(\A)$ reduces $\H$. A unital completely contractive map is called \textit{maximal} if it admits no non-trivial dilations. A unital completely contractive map is said to have the \textit{unique extension property} if there is a unique unital completely positive extension of $\pi$ on $\ca(\A)$ which is also multiplicative. It turns out that maximality and the unique extension property are equivalent properties for a unital completely contractive map \cite{MSbdy}. Dritschel and McCullough \cite{DrMc} have shown that any unital completely contractive representation $\pi$ of a unital operator algebra $\A$ admits a maximal dilation $\rho$, which by the unique extension property is automatically multiplicative.  (The reader may have realized that in this paragraph we have been within the realm of a unital category and so all maps are either assumed or required to be unital.) 

If $(\C , j)$ is a $\ca$-cover of a unital operator algebra $\A$, then there exists a largest ideal $\J \subseteq \C$, the \textit{Shilov ideal} of $\A$ in $(\C, j)$, so that the quotient map $\C \rightarrow \C \slash \J$ when restricted on $j(\A)$ is completely isometric. It turns out that $\cenv(\A) \simeq \C \slash \J$. A related result asserts that if $\pi: \A \rightarrow B(\H)$ is a completely isometric representation of a unital operator algebra $\A$ and $\rho$ a maximal dilation of $\pi$, then $\big( \ca \big(\rho(\A)\big), \rho\big) \simeq \cenv(\A)$. See \cite[Section 5]{Kat} for a recent exposition and proofs of these facts.

If $\A$ is a non unital operator algebra then we can describe the $\ca$-envelope of $\A$ by invoking its unitization as follows: if $\cenv(\A^1)= \big(\cenv(\A^1), j_1\big)$, then $\cenv(\A) \simeq (\C, j)$, where $\C \equiv \ca(j_1(\A))$ and $j\equiv {j_1}_{\mid \A}$. See \cite{Arvenv, DrMc, Kat} for more details.

There is another approach to the $\ca$-envelope of an operator algebra utilizing the concept of an injective envelope. Had we been using the ``injective envelope" approach, then the following result would have been immediate. Using the ``maximal representation" approach, it requires an extra argument.

\begin{lemma} \label{specialmax}
Let $\A$ be a unital operator algebra. Then there exists a unital completely isometric maximal representation $\pi : \A \rightarrow B(\H)$ which extends to a faithful $*$-representation of $\cenv(\A)$.
\end{lemma}

\begin{proof}
By the  Dritschel and McCullough result \cite{DrMc} such a unital completely isometric maximal representation $\pi : \A \rightarrow B(\H)$ always exists. The issue here is proving that its extension is a \textit{faithful} representation of $\cenv(\A)$. Consider $\A$ as a subalgebra of a $\ca$-algebra $\C = \ca(\A)$. Let $\phi : \A \rightarrow B(\H)$ be a maximal representation and extend $\phi$ to a $*$-representation of $\C$. Let $\J = \ker \phi$; this is the Shilov ideal of $\A$. If $q: \C \rightarrow \C\slash \J $ is the quotient map, then the desired $\pi$ is the  map that makes the following diagram
\[
\xymatrix{& \C \slash \J \ar[dr]^{\pi}& \\
\C=\ca(\A) \ar[ur]^{q} \ar[rr]_{\phi}  & & B(\H)}
\]
commutative.
\end{proof}

We also need the following result. Part (i) is elementary. Part (ii) appears as Proposition 4.4 in \cite{Arvhyp2}. The reader will not find it difficult to fill in the details.

\begin{lemma} \label{maximalfacts}
Let $\A$ be a unital operator algebra.
\begin{itemize}
\item[(i)] If $\pi : \A \rightarrow B(\H)$ is a maximal map and $\alpha \in \Aut \A$ a completely isometric automorphism then $\pi \circ \alpha$ is also maximal.
\item[(ii)] If $\pi_i : \A \rightarrow B(\H)$, $ i \in \bbI$, are maximal maps, then $\oplus_{i \in \bbI} \pi_i$ is also maximal.
\end{itemize}
\end{lemma}

We have arrived at one of the central results of the paper. Its proof is partly inspired by the proof of \cite[Theorem 3.1]{Dun1}. Note that in the case where $\G$ is amenable, the result below  resolves \cite[Problem 1]{KR}.
 \begin{theorem} \label{discrenv}
 Let $(\A, \G, \alpha)$ be a discrete dynamical system and assume that $\A$ has a contractive approximate unit $\{ e_i\}_{i \in \bbI}$ consisting of selfadjoint operators. Then
 \[
 \cenv(\A \cpr) \simeq \cenv(\A) \cpr.
 \]
  In particular if $\G$ is discrete and amenable then
  \[
 \cenv(\A \cpf) \simeq \cenv(\A) \cpf.
 \]
 \end{theorem}

\begin{proof}
Let $\A^1$ be the unitization of $\A$. By Lemma~\ref{specialmax} there exists a faithful $*$-representation $\pi: \cenv(\A^1) \rightarrow B(\H)$  whose restriction $\pi\mid_{\A^1}$ on $\A^1$ is a maximal map and therefore satisfies the unique extension property.  Let $( \overline{\pi}, \lambda_{\H})$ be the regular covariant representation for $(\cenv(A^1), \G, \alpha)$ corresponding to $\pi$. 

We claim that the restriction  $\overline{\pi} \rtimes \lambda_{\H}\mid_{(\A\cpr)^1}$ of  $$\overline{\pi} \rtimes \lambda_{\H}: \cenv(\A^1) \cpr \longrightarrow B\big(\H \otimes l^2(\G)\big)$$ satisfies the unique extension property.

Indeed start by noticing that $\overline{\pi} = \oplus_{g \in \G} \pi\circ \alpha_g$ and so Lemma~\ref{maximalfacts} implies that $\overline{\pi}$ is a maximal representation of $\A^1$. Now let $\ca\big( (\A\cpr)^1\big)$ denote the $\ca$-subalgebra of $\cenv(\A^1) \cpr $ generated by $(\A\cpr)^1$ and note that
\[
\ca\big( (\A\cpr)^1\big)= \ca\big( \A\cpr\big) ^1= \big( \cenv(\A) \cpr\big)^1.
\]
Let
\[
\rho\colon \ca\big( (\A\cpr)^1\big) \longrightarrow  B\big(\H \otimes l^2(\G)\big)
\]
be a completely positive map extending $\overline{\pi} \rtimes \lambda_{\H}\mid_{(\A\cpr)^1}$. Since $\A^1\subseteq (\A\cpr)^1$ and $$\rho\mid_{\A^1}= \overline{\pi} \rtimes \lambda_{\H}\mid_{\A^1}=\overline{\pi}\mid_{\A^1},$$ the unique extension property of $\overline{\pi}$ implies that $\rho(c) =\overline{\pi}(c)$ for any $c \in \cenv(\A^1)= \cenv(\A)^1\subseteq  \cenv(\A^1) \cpr$.

On the other hand, let $\ca( e_i\G)$ denote the $\ca$-algebra generated by all products of the form $e_i u$, with $i \in \bbI$ and $u \in \ca_{r}(\G) \subseteq  \cenv(\A^1) \cpr $. Since $\ca( e_i\G)^1 \subseteq (\A \cpr)^1$, we have $\rho(e_iu) = \overline{\pi} \rtimes \lambda_{\H}(e_iu)$. Furthermore, the containment $\ca( e_i\G)^1 \subseteq (\A \cpr)^1$ implies that $\rho$ is multiplicative on $\ca( e_i\G)^1$. Hence from \cite[1.3.12]{BlLM} or \cite[Theorem 3.18]{Paulsen}  we have  $$\rho(ce_iu)=\rho(c)\rho(e_iu) =  \overline{\pi} (c) \overline{\pi} \rtimes \lambda_{\H}(e_iu)= \overline{\pi} \rtimes \lambda_{\H}(ce_iu)$$ for all $c \in \cenv(\A)$. However $\{e_i\}_{i \in \bbI}$ is also an approximate unit for $\cenv(\A)$ and so $\rho(cu) =  \overline{\pi} \rtimes \lambda_{\H}(cu)$, for all $c \in \cenv(\A)$ and $u \in \ca_r(\G)$. This suffices to prove the claim.

\vspace{.1in}

Since $\overline{\pi} \rtimes \lambda_{\H}\mid_{(\A\cpr)^1}$ satisfies the unique extension property and its unique extension $\overline{\pi} \rtimes \lambda_{\H}$ on
$\ca\big( (\A\cpr)^1\big)= \big( \cenv(\A) \cpr\big)^1$ is faithful, we conclude that
\[
\cenv\big((\A \cpr )^1\big)\simeq \big( \cenv(\A) \cpr\big)^1.
\]
Since the $\ca$-algebra generated by $\A \cpr \subseteq \big( \cenv(\A) \cpr\big)^1$ equals $ \cenv(\A) \cpr$, we are done.
\end{proof}

 We would like to extend the previous result to the full crossed product. We do this for a large class of operator algebras which we now define.

 \begin{definition} \label{defn;hyperrigid}
 An operator algebra $\A$ is said to be \textit{hyperrigid} iff the restriction of any non-degenerate $*$-representation of $\cenv(\A)$ to $\A$ is a maximal representation. 
 \end{definition}

Since the definition of hyperrigidity allows for non-unital operator algebras, we need to explain what we mean by a maximal map in this case. A completely contractive map $\pi: \A \rightarrow B(\H)$ is said to be maximal iff all of its dilations reduce $\overline{\pi(\A)(\H)}$. Note that the subspace $\overline{\pi(\A)(\H)}$ is reducing for $\pi(\A)$ because $\A$ has a contractive approximate unit. 

The term ``hyperrigid" originates from Arveson's paper \cite{Arvhyp2} but the concept itself has already appeared in earlier works, at least as early as \cite{MS}. Arveson approaches hyperrigidity as a property of the inclusion of an operator algebra $\A$ inside a specific $\ca$-cover. Nevertheless, whenever hyperrigidity happens that $\ca$-cover has to be the $\ca$-envelope of $\A$.  Also note that our Definition \ref{defn;hyperrigid} coincides with what Arveson proves in \cite[Theorem 2.1(iii)]{Arvhyp2} as an equivalent formulation of hyperrigidity. Note here that for the verification of the  hyperrigidity of an operator algebra $\A$, we only need to examine restrictions of \textit{faithful} $*$-representation of $\cenv(\A)$.

 The list of hyperrigid algebras includes some of the fundamental examples of the theory. All Dirichlet operator algebras are hyperrigid. Peters' semicrossed products of separable $\ca$-algebras \cite{Pet} and the quiver algebras of Muhly and Solel for row finite graphs  \cite{MS} are all hyperrigid algebras; see \cite{Dun1, Kak} for a proof. Also the tensor algebras of automorphic multivariable systems are hyperrigid \cite{Kak}.
Direct limits and free products with amalgamation of hyperrigid algebras are also seen to be hyperrigid \cite{DavFK, Dun1}. Finally with Theorem~\ref{hr} we add more algebras to the list of hyperrigid algebras. 

 \begin{theorem} \label{discrenvful}
 Let $(\A, \G, \alpha)$ be a discrete dynamical system and assume that $\A$ has a contractive approximate unit consisting of selfadjoint operators. If $\A$ is hyperrigid, then
 \[
 \cenv\big(\A\rtimes_{\cenv(\A), \alpha} \G \big) \simeq \cenv(\A) \cpf.
 \]
 \end{theorem}

 \begin{proof} Assume first that $\A$ s unital. Let $(\pi, u)$ be the universal covariant representation of $(\cenv(\A), \G, \alpha)$. We claim that the restriction of $\pi\rtimes u$ to $\A\rtimes_{\cenv(\A), \alpha} \G$ has the unique extension property. This will suffice to prove the result in the unital case.

  The proof now is similar to that of Theorem~\ref{discrenv}. Indeed let
\[
\rho \colon \cenv(\A) \cpf \longrightarrow B(\H)
\]
be a completely contractive map agreeing with $\pi \rtimes u$ on $\A\rtimes_{\cenv(\A), \alpha} \G  $. Note that the restriction of $\pi$ on $\A$ is a maximal map because $\A$ is hyperrigid and so it has the unique extension property. Therefore $\pi(c) = \rho(c)$, for any $c \in \cenv(\A)$. Also $\rho(u_g)= \pi\rtimes u (u_g)$, for all $g \in \G$. Hence an application of \cite[1.3.12]{BlLM}, as in the proof of Theorem~\ref{discrenv}, proves that $\rho$ and $\pi$ agree on all of $\cenv(\A) \cpf$ and we are done.

For the non-unital case, let $(\pi, u)$ be the universal covariant representation of $\cenv(\A) \cpf$ acting on a Hilbert space $\H$; we may assume that $\pi$ is non-degenerate. Since $\A$ is hyperrigid, $\pi\mid_{\A}$ is a maximal map.

 If $\pi^1: \cenv(\A)^1\rightarrow B(\H)$ denotes the unitization of $\pi$, then we claim that $\pi^1\mid _{\A^1}$ is also maximal. Indeed if $\rho$ is a dilation of $\pi^1\mid _{\A^1}$, then the maximality of $\pi\mid_{\A}$ implies that $\rho(\A)$ (and so $\rho(\A^1)$) reduces $\overline{\pi(\A)(\H)}$. However $\pi$ is non-degenerate and since $\A$ has a contractive approximate unit, $\pi\mid_{\A}$ is also non-degenerate, i.e., $\overline{\pi(\A)(\H)} = \H$.  Hence $\rho(\A^1)$ reduces $\H$ and so $\pi^1\mid _{\A^1}$ is maximal as a \textit{unital} map. 
 
 The rest of the proof follows now familiar lines. Since $\pi^1\mid _{\A^1}$ is a maximal map, it has the unique extension property. Arguing as in the unital case above, one can show that $\pi^1\rtimes u\mid_{\A^1 \rtimes_{\cenv(\A), \alpha} \G}$ is a maximal map. Subsequently an argument identical to that of the claim in the proof of Theorem~\ref{discrenv} shows that $\pi^1\rtimes u \mid_{(\A \rtimes_{\cenv(\A), \alpha} \G)^1}$ is also maximal. This suffices to show that 
 \[
  \cenv\big(\A\rtimes_{\cenv(\A), \alpha} \G \big)^1 \simeq \big(\cenv(\A) \cpf\big)^1
  \]
  and the conclusion follows.
\end{proof}

%%%%%%%%%%%%%%%%%%%%%%%%%%%%%%%%%%%%%%%%%%\%%
%%%%%%%%%%%%%%%%%%%%%%%%%%%%%%%%%%%%%%%%%%%%%%

\section{The Hao-Ng isomorphism for discrete groups}

Let $(X, \C, \phi_X)$ be a non-degenerate $\ca$- correspondence over a $\ca$-algebra $\C$. (Whenever there is no source of confusion, the symbol $\phi_X$ will be suppressed.) There exist two important $\ca$-algebras associated with $(X, \C)$: the Cuntz-Pimsner $\ca$-algebra $\O_X$ and the Cuntz-Pimsner-Toeplitz $\ca$-algebra $\T_X$. Both are generated as $\ca$-algebras by a (unitarily equivalent) copy of $(X,C)$ that they contain; in general these algebras are not isomorphic. (See \cite{Lance} for the basics on $\ca$-correspondences and \cite{Kat, Katsura} for the precise definitions of $\O_X$ and $\T_X$.)

There is also a non-selfadjoint operator algebra associated with $(X, \C)$, the \textit{tensor algebra} $\T_X^+$ of Muhly and Solel \cite{MS}. This is the \textit{non-selfadjoint} subalgebra of $\T_X$ generated by the ``natural" copy of $(X, \C)$ that it contains. As it turns out, $\T_X^+$ is completely isometrically isomorphic to the non-selfadjoint subalgebra of $\O_X$ generated by the copy of $(X, \C)$ contained naturally in $\O_X$. This is of importance in this paper as it allows us to move from the one $\ca$-algebra to the other by staying ``inside" the non-selfadjoint subalgebra. Equally important here is a result of Kribs and the author that identifies $\O_X$ as the $\ca$-envelope of $\T^+_X$ \cite[Theorem 3.7]{KatsoulisKribsJFA}. 

Let $(X, \C)$ be a non-degenerate $\ca$- correspondence over a $\ca$-algebra $\C$. Consider an action $\alpha: \G \rightarrow \Aut \T_X$ so that $\alpha_{g}(\C)= \C$ and $\alpha_{g}(X)= X$, for all $g \in \G$. We call such an action $\alpha$ a \textit{generalized gauge action}. (Note that we do not insist that the automorphisms $\alpha_g$ fix $\C$ or $X$ elementwise but instead that they fix them only as sets.) Clearly the action $\alpha$ restricts to a generalized gauge action $\alpha : \G \rightarrow \Aut \T^+_X$, which in turn extends to a generalized gauge action on $\O_X$, because $\cenv(\T_X^+) = \O_X$. The crossed product of $\O_X$ by such actions play an important role in $\ca$-algebra theory: in the case of a Cuntz or a Cuntz-Krieger $\ca$-algebra, examples of such actions are the so-called \textit{quasi-free} actions whose crossed products have been studied extensively \cite{EF, Katsura02, KK, LN}. 

Let $(X, \C)$ be a non-degenerate $\ca$-correspondence and let $\alpha: \G \rightarrow \Aut \T_X$ be a generalized gauge action. Let $\alpha_1$ and $\alpha_2$ be the restrictions of $\alpha$ on the faithful copies of $\C\subseteq \T_X$ and $\X\subseteq \T_X$ respectively. It is easy to see that we have now a dynamical system $(\C, \G, \alpha_1)$, a group action $\alpha_2: \G \rightarrow \Aut X$ so that  for any $g \in \G$, $x, y \in X$ and $c \in \C$, 
\begin{itemize}
\item[(i)] $\alpha_{1,g}\big(\sca{x,y}\big) =\sca{\alpha_{2, g}(x),\alpha_{2, g}(y)}$
\item[(ii)] $\alpha_{2,g}(xc)=\alpha_{2,g}(x)\alpha_{1,g}(c)$
\item[(iii)]$\alpha_{2,g}(\phi_X(c)x)=\phi_X(\alpha_{1,g}(c))\alpha_{2,g}(x)$,
\end{itemize}
under the appropriate identifications. In general, any pair $(\alpha_1, \alpha_2)$ with $(\C, \G, \alpha_1)$ a dynamical system, $\alpha_2: \G \rightarrow \Aut X$ a group homomorphism and  $(\alpha_1, \alpha_2)$ satisfying the above properties is called an \textit{action} of $\G$ on $(X,C)$. (See \cite[Definition 2.1]{HN}.) It is a consequence of the universality of $\T_X$ that \textit{any} action $(\alpha_1, \alpha_2)$ of $\G$ on $(X, \C)$ comes from a generalized gauge action $\alpha: \G \rightarrow \Aut \T_X$ exactly as above. (See also \cite[Lemma 2.6]{HN} for a similar statement with $\O_X$.) In the sequel we will not be distinguishing between the concept of an action of a discrete group $\G$ on a $\ca$-correspondence $(X,\C)$  and the concept of a generalized gauge action of $\G$ on the ambient $\ca$-algebras. We will simply talk about an action of $\G$ on $(X, \C)$ and we will write $\alpha\colon \G \rightarrow (X, \C)$. 

Let $\alpha \colon\G \rightarrow (X,\C)$ be a group action. We define a $\ca$-correspondence $(X\rtimes_{\alpha}^r \G, \C \rtimes_{\alpha}^{r} \G)$ as follows. Identify formal (finite) sums of the form $\sum_g x_g u_g$, $x_g \in X$, $g \in \G$, with their image in $\O_X \cpr$ under $\pi \rtimes \lambda$, where $\pi$ is a faithful representation of $\O_X$. We call the collection of all such sums $\big(X\rtimes_{\alpha}^r \G\big)_0$. This allows a left and right action on $\big(X\rtimes_{\alpha}^r \G\big)_0$ by $\big(\C \cpr \big)_0$, i.e., finite sums of the form $\sum_{g} c_g u_g \in \C\cpr $, simply by multiplication. The fact that $\alpha$ is a gauge action guarantees that
\[
\big(\C \cpr \big)_0  \big(X\rtimes_{\alpha}^r \G\big)_0 \big(\C \cpr \big)_0  \subseteq \big(X\rtimes_{\alpha}^r \G\big)_0.
\]
Equip $\big(X\rtimes_{\alpha}^r \G\big)_0$ with the $\big(\C \rtimes_{\alpha}^r \G\big)_0$-valued inner product $\sca{.,.}$ defined by $\sca{S, T} \equiv S^*T$, with $S, T \in \big(X\rtimes_{\alpha}^r \G\big)_0$. The completion of $\big(X\rtimes_{\alpha}^r \G\big)_0$ with respect to the norm coming from $\sca{.,.}$ becomes a $(\C\rtimes_{\alpha}^r \G)$-correspondence denoted as $X\rtimes_{\alpha}^r \G$. It is not difficult to see that our definition of the $\ca$-correspondence $(X\rtimes_{\alpha}^r \G, \C \rtimes_{\alpha}^{r} \G)$ coincides with the one appearing in the selfadjoint literature \cite[p. 1082]{BKQR}.

The following appears as Theorem 7.7 in our paper \cite{KR}. However, \cite[Theorem 7.7]{KR} is proven only for unital $\ca$-correspondences and its proof requirers some minor modifications to work in the general case. We included it here for completeness and for the reader's convenience.

\begin{theorem} \label{HNtensor2}
Let $\G$ be a discrete group acting on a non-degenerate $\ca$-correspondence $(X, \C)$. Then
\[
\T^+_{X} \cpr \simeq \T^+_{X\cpr}.
\]
Therefore, $$\cenv\big (\T^+_{X} \cpr \big) \simeq  \O_{X\cpr}.$$
\end{theorem}

\begin{proof} Before embarking with the proof recall the concept and the associated notation regarding a \textit{regular} representation for the reduced crossed product $\T_X \cpr$: if $\pi$ is a $*$-representation of $\T_X$ on $\K$, then $\overline{\pi}$ will denote the representation
\[
\T_X \ni s \longmapsto \sot \sum_g \pi\big(\alpha_g^{-1}(s) \big)\otimes e_{g, g} \in B\big(\K \otimes l^2(\G)\big)
\]
($e_{p,q}$ denotes the rank-one isometry on $l^2(\G)$ that maps the basis vector $\xi_q$ on $\xi_p$, $p, q \in \G$) and $\overline{\pi}\rtimes \lambda$ will denote the associated regular representation i.e,, the integrated form of the covariant representation $(\overline{\pi}, \id\otimes \lambda)$, where $\lambda$ is the left regular representation of $\G$.

 Now with the proof. Because of \cite[Corollary 3.16]{KR} all relative reduced crossed products coincide and so we have flexibility in choosing which manifestation of $\T^+_X\cpr$ to work with. We choose the ``natural" one $\T^+_X \rtimes^r_{\T_X, \alpha}\G \subseteq \T_X \cpr$.

Notice that the $\ca$-algebra $\T_X$ contains a unitarily equivalent copy of the $\ca$-correspondence $(X, \C)$ and for the rest of the proof we envision  $(X, \C)$ as a subset of $\T_X$. Similarly the $\ca$-algebra $\T_X \cpr$ contains a (unitarily equivalent) copy of $(X\rtimes_{\alpha}^r \G, \C \rtimes_{\alpha}^{r} \G)$. Indeed $\T_X \cpr$ contains naturally a faithful copy of $\C \cpr$ and so the map
\[
\O_X \cpr \supseteq (X\rtimes_{\alpha}^r \G)_0 \ni \sum_g x_gu_g \longmapsto  \sum_g x_g u_g \in \T_X \cpr
\]
extends to a unitary equivalence of $\ca$-correspondences that embeds $(X\rtimes_{\alpha}^r \G, \C \rtimes_{\alpha}^{r} \G)$ inside $\T_X \cpr$.

Let $\rho \colon \T_X \rightarrow B(\H)$ be some faithful $*$-representation and let $V$ be the forward shift acting on $l^2(\bbN)$. The map
\begin{align*}
\C\ni &c \longmapsto \rho(c) \otimes I \in B(\H \otimes l^2(\bbN)) \\
X\ni &x \longmapsto \rho(x) \otimes V  \in B(\H \otimes l^2(\bbN))
\end{align*}
is a Toeplitz representation of $(X, \C)$ that admits a gauge action and and satisfies the requirements of Katsura's Theorem \cite[Theorem 6.2]{Katsura}. Therefore it establishes a faithful representation $\pi : \T_X \rightarrow B(\H \otimes l^2(\bbN))$.

Now view the regular representation $\overline{\pi} \rtimes \lambda_{\H \otimes l^2(\bbN) }$ as a representation of the $\ca$-correspondence $(X\rtimes_{\alpha}^r \G, \C \rtimes_{\alpha}^{r} \G)$. Since
\begin{align*}
\big( \C \rtimes_{\alpha}^{r} \G \big)_0 \ni \sum_{g} c_{g}u_{g} \ \ \longmapsto \ \ &(\overline{\pi} \rtimes \lambda )\big( \sum_{g} c_{g}u_{g} \big) \\
&=\sum_g\sum_h \rho\big(\alpha_h^{-1}(c_g) \big)\otimes I \otimes e_{h, g^{-1}h} \\
\big( X \rtimes_{\alpha}^{r} \G \big)_0 \ni \sum_{g} x_{g}u_{g} \ \ \longmapsto \ \ &(\overline{\pi} \rtimes \lambda )\big( \sum_{g} x_{g}u_{g} \big) \\
&=\sum_g\sum_h \rho\big(\alpha_h^{-1}(x_g) \big)\otimes V \otimes e_{h, g^{-1}h} ,
\end{align*}
the above extends to an isometric representation of $(X\rtimes_{\alpha}^r \G, \C \rtimes_{\alpha}^{r} \G)$ that admits a gauge action (because of the middle factor $V$) and satisfies the requirements of Katsura's Theorem \cite[Theorem 6.2]{Katsura}. Hence its integrated form is a canonical faithful representation of the Toeplitz-Cuntz-Pimsner algebra $\T_{X\cpr}$. In other words, if $(\pi_{\infty}, t_{\infty})$ is the universal Toeplitz representation of $(X\rtimes_{\alpha}^r \G, \C \rtimes_{\alpha}^{r} \G)$, then there exists $*$-isomorphism  
\[
\phi \colon \ca(\pi_{\infty}, t_{\infty}) \longrightarrow \overline{\pi}\rtimes \lambda\big( \T_X \cpr\big)
\]
satisfying 
\[
\phi \Big( \pi_{\infty}\big( \sum_{g} c_{g}u_{g} \big)\Big) = (\overline{\pi} \rtimes \lambda )\big( \sum_{g} c_{g}u_{g} \big) , \mbox{ for all }\sum_{g} c_{g}u_{g} 
 \in\big( \C \rtimes_{\alpha}^{r} \G \big)_0 
 \]
 and 
 \[
\phi \Big( t_{\infty}\big( \sum_{g} x_{g}u_{g} \big)\Big) =  (\overline{\pi} \rtimes \lambda )\big( \sum_{g} x_{g}u_{g} \big), \mbox{ for all }\sum_{g} x_{g}u_{g} 
 \in\big( X \rtimes_{\alpha}^{r} \G \big)_0 
 \]
Since $\pi$ is faithful, $\overline{\pi}\rtimes \lambda$ is a faithful representation of $\T_{ X}\cpr$ and so $(\overline{\pi} \rtimes \lambda)^{-1}\circ \phi$ establishes a $*$-isomorphism from $\ca(\pi_{\infty}, t_{\infty}) \simeq \T_{ X\cpr}$ onto $\T_X\cpr$ that maps $\T^+_{X\cpr}$ onto $\T^+_{X} \cpr$ in a canonical way. Hence $\T^+_{X} \cpr \simeq \T^+_{X\cpr}$.

Finally the isomorphism $\cenv\big (\T^+_{X} \cpr \big) \simeq  \O_{X\cpr}$ follows from
\cite[Theorem 3.7]{KatsoulisKribsJFA}, which implies the identification $\cenv(\T^+_{X \cpr}) \simeq \O_{X \cpr}$.
\end{proof}

We now use the above to obtain our generalization of the Hao-Ng Theorem applicable to all discrete groups. In the case where $\G$ is amenable our result below is just the Hao-Ng Theorem for discrete groups \cite[Theorem 2.10]{HN}.  In the case where $\G$ is an exact discrete group, the result below was obtained in \cite[Theorem 5.5]{BKQR} and was highlighted as one of the central results of that paper that actually explained a remark of Katsura appearing in \cite{HN}. In \cite{Kim} more general groups were allowed but with restrictions on the $\ca$-correspondences considered. We now remove all conditions, apart from the discreteness of $\G$.

\begin{theorem} \label{HaoNgreduced}
Let $\G$ be a discrete group acting on a non-degenerate $\ca$-correspondence $(X, \C)$. Then
\[
\O_X \cpr \simeq\O_{X\cpr}.
\]
  \end{theorem}
  
  \begin{proof} 
  By Theorem~\ref{HNtensor2} we have $\cenv\big (\T^+_{X} \cpr \big) \simeq  \O_{X\cpr}$. On the other hand, Theorem~\ref{discrenv} implies that $\cenv\big (\T^+_{X} \cpr \big) \simeq \cenv (\T^+_{X}) \cpr  $. Hence $\cenv (\T^+_{X}) \cpr  \simeq \O_{X\cpr}$. Finally \cite[Theorem 3.7]{KatsoulisKribsJFA} shows that $\cenv (\T^+_{X})\break \simeq \O_X$ and we are done.
  \end{proof}

\begin{remark}
There is a subtle point in the proof of Theorem~\ref{HaoNgreduced} that should not go unnoticed. The isomorphism $\cenv\big (\T^+_{X} \cpr \big) \simeq  \O_{X\cpr}$ appearing in the proof of Theorem~\ref{HaoNgreduced} follows from considering $\T^+_X \cpr$ as a subalgebra of the crossed product $\ca$-algebra $\T_X\cpr$; this was actually mentioned in the proof of Theorem~\ref{HNtensor2}. On the other hand, an inspection of the proof of Theorem~\ref{discrenv} shows that the other isomorphism $\cenv\big (\T^+_{X} \cpr \big) \simeq \cenv (\T^+_{X}) \cpr  $ appearing in in the proof of Theorem~\ref{HaoNgreduced} follows by considering  $\T^+_X \cpr$ as a subalgebra of the crossed product $\ca$-algebra $\cenv(\T_X^+)\cpr \simeq \O_X\cpr$. Even though in general $\T_X\cpr$ and $\O_X \cpr$ are not isomorphic as $\ca$-algebras, it is the content of \cite[Theorem 3.12]{KR} that allows us to identify the two non-selfadjoint subalgebras discussed above.
\end{remark}

  We now want to obtain the analogous result for the full crossed product. We will not be able to do this for all $\ca$-correspondences but instead for a large class which includes all row-finite graph $\ca$-correspondences, Hilbert $\ca$-bimodules etc, which we now describe.
  
  \begin{definition}
  A $\ca$-correspondence $(X, \C)$ is said to be \textit{hyperrigid} if the corresponding tensor algebra $\T_X^+$ is hyperrigid as an operator algebra.
  \end{definition}
  
  Many of the algebras appearing below Definition~\ref{defn;hyperrigid} are actually tensor algebras of $\ca$-correspondences, thus providing examples of hyperrigid $\ca$-correspondences. In particular the Dirichlet tensor algebras coincide with the tensor algebras of Hilbert $\ca$-bimodules \cite{Kak}; hence the Hilbert $\ca$-bimodules are hyperrigid. Peters's semicrosed products and the quiver algebras of Muhly and Solel are tensor algebras of $\ca$-correspondences. (See \cite[Section 3]{Kat} for an detailed description of these $\ca$-correspondences. The later class of tensor algebras is associated with all graph $\ca$-correspondences.) Also the tensor algebras for multivariable systems are tensor algebras of $\ca$-correspondences \cite[Chapter 2, Section 3]{DavKatMem} and so in the automorphic case they generate new examples of hyperrigid $\ca$-correspondences. Additional examples of hyperrigid $\ca$-correspondences/tensor algebras have also appeared recently in the important works \cite{DavFK, KakS}.
  
  In order to provide our version of the Hao-Ng isomorphism for the full crossed product we also need a suitable $\ca$-correspondence; for that we follow \cite{KR}. Identify both $\big(X\rtimes_{\alpha} \G\big)_0$ and $\big( \C \cpf \big)_0$ with their natural images inside $\O_X \cpf$ this time. This allows again a left and right action on $\big(X\rtimes_{\alpha} \G\big)_0$ by $\big(\C \cpf \big)_0$ simply by multiplication. Equip $\big(X\rtimes_{\alpha} \G\big)_0$ with the $\C \hat{\rtimes}_{\alpha} \G$-valued inner product $\sca{.,.}$ defined by $\sca{S, T} \equiv S^*T$, $S, T \in \big(X\rtimes_{\alpha} \G\big)_0$, where $\C \hat{\rtimes}_{\alpha} \G$ denotes the $\ca$-subalgebra of $\O_X \cpf$ generated by $\big( \C \cpf \big)_0$. The completion of $\big(X\rtimes_{\alpha} \G\big)_0$ with respect to the norm coming from $\sca{.,.}$ becomes a $\C \hat{\rtimes}\G$-correspondence denoted as $X\hat{\rtimes}_{\alpha} \G$.
  
Finally we need an analogue of Theorem~\ref{HNtensor2} for the full crossed product. This appears as Theorem 7.7.(i) in \cite{KR}. This time the proof requires some extra arguments based on a result that we did coin as the Extension Theorem in \cite{KR}. We direct the reader to \cite{KR} for the proof and further details.
  
 \begin{theorem} \label{HNtensor1}
Let $\G$ be a discrete group acting on a non-degenerate $\ca$-correspondence $(X, \C)$. Then
 \[
\T^+_{X} \rtimes_{\O_X ,\alpha} \G \simeq \T^+_{X\cpd}\quad \mbox{  and  } \quad \cenv\big (\T^+_{X } \rtimes_{\O_X ,\alpha} \G \big) \simeq  \O_{X\cpd}
 \]
\end{theorem}

By assembling Theorem~\ref{discrenvful} and Theorem~\ref{HNtensor1} exactly as in the proof of Theorem ~\ref{HaoNgreduced}, we obtain

\begin{theorem} \label{HaoNgfull}
Let $\G$ be a discrete group acting on a non-degenerate hyperrigid $\ca$-correspondence $(X, \C)$. Then
\[
 \O_X \cpf \simeq \O_{X\cpd}.
\]
  \end{theorem}
  
  In particular, Theorem~\ref{HaoNgfull} applies to the $\ca$-correspondence $X_{G}$ of a row-finite graph $G$, since the tensor algebra $\T_{X_G}^+$ is just one of the quiver algebras of Muhly and Solel \cite{MS}. Note however that Muhly and Solel show in \cite{MS} that the tensor algebra of the graph with one vertex and infinitely many edges fails to be hyperrigid. We therefore wonder whether Theorem~\ref{HaoNgfull} remains valid for an arbitrary graph $\ca$-correspondence.
  
  One final word. It is easy to see that in the case where $\G$ is amenable our $\ca$-correspondence $(\C \hat{\rtimes}\G , X\hat{\rtimes}_{\alpha} \G)$ coincides with the $\ca$-correspondence $(\C\cpf, X \cpf)$ as defined in \cite[p. 1082]{BKQR}. We do not know if this remains true in the generality of Theorem~\ref{HaoNgfull}.
  %%%%%%%%%%%%%%%%%%%%%%%%%%%%%%%%%%%%%%%
  %%%%%%%%%%%%%%%%%%%%%%%%%%%%%%%%%%%%%%%%%
  \section{More on Hyperrigidity and crossed products}
  
  The purpose of this last section is to further highlight the interplay between hyperrigidity and the crossed product theory of \cite{KR} by proving the following.
  
  \begin{theorem} \label{hr}
 Let $(\A, \G, \alpha)$ be a unital and discrete dynamical system. Then $\A$ is hyperrigid if and only if $\A\cpr$ is hyperrigid.
 \end{theorem}
 
 \begin{proof}
 Assume that $\A$ is hyperrigid. Let $\rho\colon \A \cpr \rightarrow B(\H)$ be a representation of $\A$ that extends to a $*$-representation of $\cenv\big( \A \cpr\big) \simeq \cenv(\A) \cpr $. We are to prove that $\rho$ is a maximal map. Note that $\rho\mid_{\A}$ is a maximal map because $\A$ is hyperrigid.
 
 Let $\pi' \colon \A \cpr \rightarrow B(\K')$ be a dilation of $\rho$. To show that $\rho$ is a maximal map means that we have to prove that $\pi'$ is a trivial dilation, i.e., $\pi'(\A\cpr)$ reduces $\H$. By Dritchell-McCullough further dilate $\pi'$ to a maximal map 
 \[
 \pi \colon \A \cpr \longrightarrow B(\K).
 \] It suffices to show that $\pi(\A\cpr)$ reduces $\H$. Note that by maximality $\pi$ extends to a $*$-representation of $\cenv(\A) \cpr $. In particular $\pi$ is multiplicative and so $\H$ is semi-invariant for $\pi\big(\A\cpr)$.
 
Since $\pi\mid_{\A}$ dilates $\rho\mid_{\A}$, which is a maximal map, we obtain that $\pi(\A)$ reduces $\H$. On the other hand, $\H$ is semi-invariant for $\pi\big(\A\cpr)$ and in particular for the $\ca$-algebra $\pi\big(\ca_r(\G)\big)$. Hence $\pi\big(\ca_r(\G)\big)$ reduces $\H$ and since $\pi$ is multiplicative, $\pi(\A\cpr)$ reduces $\H$, as desired.

Assume conversely that $\A\cpr $ is hyperrigid. Let $\rho\colon \A  \rightarrow B(\H)$ be a representation of $\A$ that extends to a $*$-representation of $\cenv (\A) $. We are to prove that $\rho$ is a maximal map. 

Assume that $\pi\colon \A  \rightarrow B(\K)$ is a dilation of $\rho$. We need to show that $\pi(\A)$ reduces $\H$. Without loss of generality we can assume that $\pi$ is a maximal dilation of $\rho$. (See the second paragraph of the proof.) Notice now that $\overline{\pi}\rtimes \lambda_{\K}$ dilates $\overline{\rho}\rtimes \lambda_{\H}$. However $\A \cpr $ is hyperrigid and $\overline{\rho}\rtimes \lambda_{\H}$ extends to a $*$-representation of $\cenv(\A) \cpr $. Hence $\overline{\rho}\rtimes \lambda_{\H}$ is maximal and so $\overline{\pi}\rtimes \lambda_{\K}(\A \cpr)$ reduces $\H \otimes l^2(\G)$. From this it is easy to see that $\pi(\A)$ reduces $\H$ and so $\rho$ is maximal.
 \end{proof}

Of course an analogous result holds for the full crossed product. Indeed

  \begin{theorem} \label{hrf}
 Let $(\A, \G, \alpha)$ be a unital and discrete dynamical system. Then $\A$ is hyperrigid if and only iff $\A\rtimes_{\cenv(\A), \alpha} \G$ is hyperrigid.
 \end{theorem}
%%%%%%%%%%%%%%%%
%%%%%%%%%%%%%%%%%%%%%%%%%%%%%%%%

\end{document}